\documentclass[reqno]{amsart}
\usepackage{amssymb,amsmath,amsthm,bbm} 
\usepackage[margin=1in]{geometry}
\usepackage[colorinlistoftodos]{todonotes}
\usepackage{graphicx,ctable,booktabs}
\usepackage[mathscr]{euscript}
\usepackage{enumerate, mathdots, tikz-cd}
\usepackage{import}
\usepackage[
backend=biber,
style=alphabetic,
citestyle=alphabetic,
maxbibnames=99]{biblatex}

\def\eps{\epsilon}
\def\veps{\varepsilon}

\newcommand{\inv}[1]{{#1}^{-1}}

\newcommand{\g}[2]{\langle {#1}\,,\,{#2}\rangle}

\def\CC{{\mathbb C}}

\def\RR{{\mathbb R}}

\def\1{{\mathbbm 1}}

\def\tr{\operatorname{tr}}

\def\sin{\operatorname{sin}}

\def\grad{\operatorname{grad}}

\def\div{\operatorname{div}}

\newtheorem{thm}{Theorem}
\newtheorem{prop}[thm]{Proposition}
\newtheorem{lem}[thm]{Lemma}

\newtheorem{definition}{Definition}

\newcommand{\rg}[2]{( {#1},\,{#2})}
\numberwithin{equation}{section}
\numberwithin{thm}{section}

\usepackage[utf8]{inputenc}
\title{An Inverse Problem for Fractional Connection Laplacians}
\author{Chun-Kai Kevin Chien}
\date{}

\begin{document}

\maketitle

\begin{abstract}
    Consider a fractional operator $P^s$, $0<s<1$, for connection Laplacian  $P:=\nabla^*\nabla+A$ on a smooth Hermitian vector bundle over a closed, connected Riemannian manifold of dimension $n\geq 2$. We show that local knowledge of the metric, Hermitian bundle, connection, potential, and source-to-solution map associated with $P^s$ determines these structures globally. This extends a result known for the fractional Laplace-Beltrami operator.
\end{abstract}

\section{Introduction}
\subsection{Background} 
The classical Calder{\'o}n problem seeks to recover the coefficient $\gamma$ in the elliptic equation 
\begin{equation}\label{eq_calderon1}
    \div(\gamma\grad u)=0
\end{equation} 
on an open set $U\subset \RR^n$ from its Dirichlet-to-Neumann map $\Lambda$, which takes prescribed boundary values $u|_{\partial U}=f$ to $\gamma\partial_\nu|_{\partial U}u$ . As outlined in \cite{uhlmann2009electrical}, this problem is motivated by real-world applications such as electrical impedance tomography, but the resolution of this problem and its analogues have seen numerous geometric applications. For example, recovering the metric $g$ up to a conformal factor from $\Lambda$ associated with 
\begin{equation}\label{eq_calderon2}
    \Delta_gu=0
\end{equation}
instead of \eqref{eq_calderon1} in two dimensions is a critical part of the proof of boundary rigidity in \cite{pestov2005two}.

One can generalize Calder{\'o}n's problem by considering bundle valued operators. A version of this is found in \cite{albin2013inverse}, where the authors replaced \eqref{eq_calderon2} with 
\begin{equation}\label{eq_calderon3}
    \nabla^*\nabla+A=0,
\end{equation}
with connection $\nabla$ on a Hermitian bundle over a Riemann surface with boundary $E\to M$ and potential $A\in L^\infty(M;\operatorname{End}(E))$. It was shown that appropriate boundary Cauchy data fixes $\nabla$ and $A$ up to unitary isomorphism. This recently was applied by \cite{Bao_2019} to the AdS/CFT correspondence in physics, where this result was used to argue that the bulk metric is fixed by entanglement entropies of regions on the boundary CFT.

Instead of local operators, one can consider nonlocal operators such as the fractional Laplacian, whose formulation as a suitably extended Dirichlet-to-Neumann map was given in the seminal paper \cite{caffarelli2007extension}. The geometric version of this extension involving the fractional Laplace-Beltrami operator was provided in \cite{chang2011fractional} in the context of asymptotically hyperbolic manifolds, and its relation to the fractional Yamabe problem was explored in \cite{gonzalez2013fractional}. There has also been recent interest in fractional Laplacians and related fractional operators from an inverse problems perspective. In the Euclidean setting, \cite{ghosh2020calderon} showed how one recovers the $L^\infty$ potential $q$ from a  nonlocal exterior Dirichlet-to-Neumann map associated with
\begin{align}\label{eq_nonloc}
\begin{cases}
(\Delta^s+q)u=0 \, \text{ on }U,\\
u=f\, \text{ on }\RR^n\setminus \overline{U}.
\end{cases}
\end{align}
This result was extended for lower regularity potentials in \cite{ruland2020fractional}. Work of \cite{cekic2020calderon} and \cite{covi2022higher} showed similar results for more general local perturbations of the fractional Laplacian, and \cite{bhattacharyya2021inverse} considered the related inverse problem for a ``regional''-type nonlocal perturbation. Work of \cite{ghosh2017calderon} dealt with the inverse problem associated with \eqref{eq_nonloc} with $\Delta$ replaced by elliptic $-\nabla\cdot(A(x)\nabla)$.

In a geometric setting, \cite{feiz2021} recovered the Riemannian metric $g$ from a fractional source-to-solution map $\mathcal{L}^{\operatorname{frac}}$ associated with 
\begin{equation}\label{eq_og}
\Delta_g^su=f
\end{equation}
on a closed and connected manifold $M$. Their strategy, which will also be used in this paper, was to relate source-to-solution data to a local equality of heat kernels; this then transformed the problem into an inverse problem for the wave equation, for which there are well-established techniques involving finite propogation speed and controllability to recover the metric structure.
Work by \cite{quan2022} used similar techniques to recover $g$ and additional Clifford bundle structures from $\mathcal{L}^{\operatorname{frac}}$ associated with a fractional Dirac operator and also discussed how such  fractional operators arise in a physical setting.

Our work is a generalization of \cite{feiz2021}, where we replace \eqref{eq_og} with a fractional power of a symmetric Laplace-type operator, $P^s$, to be defined below. The operator $P$ has the form of a lower order perturbation of a bundle Laplacian, and we show that a suitably-defined source-to-solution map $\mathcal{L}^{\operatorname{frac}}$ determines the geometric structures on a Hermitian bundle $E\to M$, or equivalently, the bundle-valued coefficients of $P$ up to gauge.

\subsection{Fractional connection Laplacians and source-to-solution maps}
Let $(M,g)$ be a smooth, closed Riemannian manifold equipped with a smooth Hermitian vector bundle $(E,\g{\cdot}{\cdot}_E)$.  Recall that if $\tilde E$ is an arbitrary tensor product $E^{\otimes m}\otimes T^{k,l}M$, then the appropriate tensor product of $ \g{\cdot}{\cdot}_E$ and $g$ together determine a bundle metric $\g{\cdot}{\cdot}_{\tilde E}$ on $\tilde E$. This induces an $L^2$-product on smooth sections $C^\infty(\tilde E)$ given by 
\[(u,v)_{L^2(\tilde E)}:=\int_M \langle u,v\rangle_{\tilde E} \, dV_g ,\]
and $L^2(M;\tilde E)$ is the completion of $C^\infty(M;\tilde E)$ under the induced norm. If $U\subset M$ is an open set, we denote the compactly supported smooth sections of $\tilde E$ inside $U$ by $C_0^\infty(U;\tilde E)$. We will drop the reference to the bundle $\tilde E$ whenever it is contextually clear.

Consider any smooth connection $\nabla:C^\infty(M;E)\to C^\infty(M;E\otimes T^*M)$ compatible with the Hermitian structure on $E$, that is $d\g{u}{v}=\g{\nabla u}{v}+\g{u}{\nabla v}$.  The Sobolev spaces $H^k(M;E)$ of nonnegative integer order $k$ are given by the completion of $C^\infty(M;E)$ under the norm induced by the inner product 
\[(u,v)_{H^k}:=(u,v)_{L^2}+(\nabla u,\nabla v)_{L^2}+\cdots +(\nabla^k u,\nabla^k v)_{L^2}.\]
Since $M$ is compact, the spaces $L^2$ and $H^k$ as defined do not depend on $g$, $\g{\cdot}{\cdot}_E$, or $\nabla$, though the inner products do.

With respect to $\rg{\cdot}{\cdot}_{L^2}$, the connection $\nabla$ has formal adjoint $\nabla^*$ in the sense that $(\nabla u,v)_{L^2}=(u,\nabla^*v)_{L^2}$ for $u\in C^\infty(M;E)$, $v\in C^\infty(E\otimes T^*M)$. Their composition is the \textit{bundle Laplacian} $\nabla^*\nabla=-\tr_g\nabla^2=:\Delta^E$. By elliptic regularity, there is a unique self-adjoint extension $\Delta^E:\mathcal D(\Delta^E)\subset L^2(M;E)\to L^2(M;E)$ with domain $\mathcal{D}(\Delta^E)=H^2(E)$. Taking a symmetric potential $A\in L^\infty(M;\operatorname{End}(E))$ in the sense that $\rg{u}{Av}_{L^2}=\rg{Au}{v}_{L^2}$,  we now consider the \textit{(generalized) connection Laplacian}
\begin{equation}\label{def_P}
    P:=\Delta^E+A
\end{equation}
which is also self-adjoint with domain $\mathcal{D}(P)=\mathcal{D}(\Delta^E)=H^2(M;E)$. The spectral theorem for self-adjoint operators yields an orthogonal decomposition $L^2=\bigoplus_{k=1}^\infty V_k$ in terms of the spectrum  $-c=\lambda_1<\lambda_2<\cdots\to +\infty$, $c\geq 0$, of $P$ (which is discrete), where each $V_k:=\operatorname{Ker}(P-\lambda_k)$ is the finite dimensional eigenspace for eigenvalue $\lambda_k$. In terms of the $L^2$-projectors $\pi_k:L^2\to V_k$, we define:

\begin{definition}
Let $P$ as in \eqref{def_P} be nonnegative (i.e.\ $\lambda_1\geq 0$). The \textit{fractional connection Laplacian} $P^s:\mathcal{D}(P^s)\subset L^2\to L^2$ for $0<s<1$ is given by 
\begin{equation}\label{def_fracP}
        P^su:=\sum_{k\geq 1}\lambda_k^s\pi_ku
\end{equation}
with domain $\mathcal{D}(P^s)=H^{2s}:=\{u\in L^2:\sum_{k\geq 1}\lambda_k^{2s}\|u\|^2_{L^2}<\infty\}$.
\end{definition}

Turning our attention to the restriction of $P^s$ to $\mathcal D(P^s)\cap \operatorname{Ker}(P)^\perp$ in the case that $P$ has a zero eigenvalue, we see that $P^s$ has inverse \begin{equation}\label{eq_negfrac1}
    P^{-s}f:=\sum_{k\geq 1}\lambda_k^{-s}\pi_kf,\,\hspace{4pt}0<s<1,
\end{equation} where the sum is taken over the nonzero eigenvalues. That is, if $f\in\operatorname{Ker}(P)^\perp$, then $u=P^{-s}f$ is the unique solution to $P^su=f$ on $M$. This motivates the following:
\begin{definition}\label{def_DNfrac}
The local fractional source-to-solution map $\mathcal{L}^{\operatorname{frac}}_{P,U}:C_0^\infty(U;E)\cap \operatorname{Ker}(P)^\perp\to C^\infty(U;E)$ associated with $P^s$ over an open set $U\subset M$ is given by $f\mapsto P^{-s}f|_U$.
\end{definition}

We now examine how $\mathcal{L}^{\operatorname{frac}}_{P,U}$ behaves under pullback. Consider smooth, closed, connected Riemannian manifolds equipped with respective Hermitian vector bundles, compatible connections, and symmetric potentials:
\begin{equation}\label{eq_structure}
\mathcal M_i:=(M_i,g_i, E_i,\g{\cdot}{\cdot}_{E_i},\nabla^{E_i},A^{E_i}),\, i=1,2.
\end{equation}
\begin{definition}
Suppose we have a smooth vector bundle isomorphism
\begin{equation}\label{fig_commute}
  \begin{tikzcd}
    E_2 \arrow{r}{\Psi} \arrow{d}[swap]{\pi_2} &E_1 \arrow{d}{\pi_1} \\
    M_2 \arrow{r}{\psi} &M_1 \\
  \end{tikzcd}
\end{equation}
We say $\Psi: \mathcal M_2\to \mathcal M_1$ is a \textit{structure-preserving isomorphism} if the following conditions are satisfied:
\begin{enumerate}
    \item The map $\psi$ is an isometry: $g_2=\psi^*g_1$.
    \item The Hermitian metric $\g{\cdot}{\cdot}_{E_2}=\Psi^*\g{\cdot}{\cdot}_{E_1}$, where $\Psi^*\g{u_2}{\tilde u_2}_{E_1}=\g{u_1}{\tilde u_1}_{E_1}$ for $u_2=\Psi^*u_1$, $\tilde u_2=\Psi^* \tilde u_1$. Note that the sections $u_2\in C^\infty(M_2;E_2)$ are in bijective  correspondence with $u_1\in C^\infty(M_1;E_1)$ by the pullback $u_2=\Psi^*u_1:=\inv\Psi\circ u_1\circ \psi$.
    \item The $\g{\cdot}{\cdot}_{E_2}$-compatible connection is  $\nabla^{E_2}=\Psi^*\nabla^{E_1}$, with the pullback connection acting on $u\in C^\infty(M_2;E_2)$, $\nu\in C^\infty(TM_2)$ by 
\begin{equation*}
    (\Psi^*\nabla^{E_2})_\nu u:=\Psi^*(\nabla^{E_1}_{d\psi(\nu)}({\inv\Psi}^*u)).
\end{equation*}
    \item The symmetric potential satisfies $A^{E_2}=\Psi^*A^{E_1}:=\Psi^*A^{E_1}{\inv\Psi}^*$. 
\end{enumerate}
\end{definition}

We note that a structure-preserving isomorphism naturally arises from any diffeomorphism $\psi:M_2\to M_1$ by defining $g_2:=\psi^*g_1$, $E_2:=\psi^*E_1$, and pulling back the remaining structures using $\Psi$; in this case, the bundle isomorphism $\Psi$ restricts to the identity on each fibre.

Letting   $P_i:=\Delta^{E_i}+A^{E_i}$ for $i=1,2$, we find $P_2(\Psi^*u)=\Psi^*(P_1u)$ for $u\in C^\infty(M_1;E_1)$ and furthermore $\operatorname{Ker}(P_2-\lambda_k)=\Psi^*\operatorname{Ker}(P_1-\lambda_k)$ by looking at pullbacks of $L^2$-eigensections. For sections not in $\operatorname{Ker}(P_1)$, we therefore find $\Psi^*(P_1^{-s}u)=P_2^{-s}(\Psi^*u)$. Moreover, restricting to $f\in C_0^\infty(U_1;E_1)\cap\operatorname{Ker}(P_1)^\perp$, setting $U_2=\inv\psi(U_1)$ and letting $\mathcal{L}_i^{\operatorname{frac}}=\mathcal{L}^{\operatorname{frac}}_{P_i,U_i}$ for $i=1,2$, we have 
\begin{equation}\label{eq_pullback}
    \Psi^*\mathcal{L}_1^{\operatorname{frac}}f=\mathcal{L}_2^{\operatorname{frac}}\Psi^*f.
\end{equation}

\subsection{Main result and outline} Our discussion above implies that \eqref{eq_pullback} holds whenever $\Psi:\mathcal M_2\to\mathcal M_1$ is a structure-preserving isomorphism. This motivates our main theorem, which is the following rigidity result:

\begin{thm}\label{thm_main}
Let $\mathcal M_1, \mathcal M_2$ be smooth, closed, connected Riemannian manifolds of dimension $n\geq 2$ equipped with respective Hermitian vector bundles, compatible connections, and symmetric potentials as in \eqref{eq_structure} such that $P_1,P_2\geq 0$. For $i=1,2$, consider open sets $U_i\subset M_i$ such that there is a structure-preserving isomorphism $\mathcal M_1|_{U_1}\cong \mathcal M_2|_{U_2}$:
\begin{equation}\label{fig_commute1}
  \begin{tikzcd}
    E_2 \arrow{r}{\tilde\Psi} \arrow{d}[swap]{\pi_2} &E_1 \arrow{d}{\pi_1} \\
    U_2 \arrow{r}{\tilde\psi} &U_1 \\
  \end{tikzcd}
\end{equation}

If  $\tilde\Psi^*\mathcal{L}_1^{\operatorname{frac}}=\mathcal{L}_2^{\operatorname{frac}}\tilde\Psi^*$ on $C_0^\infty(U_1;E_1)\cap \operatorname{Ker}(P_1)^\perp$, then for any open set $\mathcal O_2\Subset U_2$ there exists a global structure-preserving isomorphism $\Psi: \mathcal M_2\to\mathcal M_1$ with $\Psi|_{\mathcal O_2}=\tilde \Psi|_{\mathcal O_2}$ and $\psi|_{\mathcal O_2}=\tilde \psi|_{\mathcal O_2}$.
\end{thm}

Our result generalizes that of \cite{feiz2021}, which concerned the trivial line bundle $\CC\to M$ with Euclidean connection and $A=0$ so that $P$ is the Laplace-Beltrami operator. As any Laplace-type operator on $E\to M$ is locally of form
\begin{equation}\label{eq_loc}
P=-g^{ij}\operatorname{Id}_E\partial_i\partial_j+b^j\partial_j+c, \,\text{ for } b^j,c\in C^\infty(U;\operatorname{End}(E)),
\end{equation}
our result can also be interpreted as recovering the coefficients of \eqref{eq_loc} up to gauge  for symmetric Laplace-type operators from local fractional source-to-solution data.

The outline of our proof---and this paper---is as follows. In Section One, we adapt the argument used in \cite{feiz2021} to show that $\mathcal{L}_{P,U}^{\operatorname{frac}}$ determines the heat kernel of $P$ on $U$. By a variant of the Kannai  transmutation formula, such a determination then yields the local source-to-solution map for a linear wave equation for $P$. We review propagation and unique continuation results for the wave equation in Section Two as well as recall some useful properties of cut times on Riemannian manifolds. Section Three proves our main theorem in two parts. The first uses the argument in \cite{helin2018correlation} to recover the Riemannian structure on $M$, while the second adapts the argument in \cite{kurylev2018inverse} to recover the remainder of the structure-preserving isomorphism.

\subsection*{Acknowledgements} The author would like to thank Gunther Uhlmann for suggesting this problem and his support throughout; and Hadrian Quan for helpful references.

\section{Transformation into a linear wave inverse problem}

\subsection{Local identification of the heat kernel}

On $\mathcal M$ and associated $P$, we start with an alternative expression for \eqref{eq_negfrac1} in terms of  the Gamma function 
\[\Gamma(s)=\int_0^{\infty} t^{s-1}e^{-t}\,dt\]
via spectral mapping. A substitution $t\mapsto t/\lambda$ 
yields 
\[\lambda^{-s}=\frac{1}{\Gamma(s)}\int_0^{\infty} t^{s-1}e^{-t\lambda}\,dt,\]
and therefore
\begin{equation}\label{eq_gamma}
    P^{-s}u=\frac{1}{\Gamma(s)}\int_0^\infty t^{s-1}e^{-tP}u\,dt, \,\hspace{4pt} 0<s<1,
\end{equation}
where $e^{-tP}$ is the heat semigroup generated by $P$. Recall that for $f\in L^2(M;E)$, the heat semigroup gives the smooth solution  to the heat equation
\begin{align}
\begin{cases}
(\partial_t+P)u(t,x)=0 \, \text{ on }(0,\infty)\times M,\\\label{eq_heat} 
u|_{t=0}=f.
\end{cases}
\end{align}
by setting $u(t,x):=e^{-tP}f$. Furthermore, its kernel satisfies the following:

\begin{lem}\label{lem_kerbounds}
On $\mathcal M$, the heat kernel $e^{-tP}(x,y)$ satisfies the pointwise bound
\begin{equation}\label{eq_ptbd}
    |e^{-tP}(x,y)|< Ct^{-(n+1)}e^{-d_g(x,y)^2/t}\,\text{ for } t<1
\end{equation}
and the operator bound
\begin{equation}\label{eq_opbd}
    \|e^{-tP}\|_{L^1\to L^\infty}\leq Ct^{-n/2}\,\text{ for } t>0.
\end{equation}
\end{lem}
\begin{proof}
The first inequality can be found in \cite{ludewig2019strong}, while the second follows from Sobolev embedding and \cite{varopoulos1985hardy}.
\end{proof}

These bounds are crucial in showing the following:
\begin{prop}\label{prop_heatker}
Let $E\to M$ be a smooth vector bundle over a closed manifold, and consider  geometric structures $(g_i,\g{\cdot}{\cdot}_i, \nabla_i,A_i)$,  $i=1,2$, on this bundle and manifold that coincide over an open set $U\subset M$. Consider the local source-to-solutions maps $\mathcal{L}^{\operatorname{frac}}_i:=\mathcal{L}^{\operatorname{frac}}_{P_i,U}$ associated with these structures. If $\mathcal{L}^{\operatorname{frac}}_1=\mathcal{L}^{\operatorname{frac}}_2$, then we have the local equality of heat kernels 
\begin{equation}
    e^{-tP_1}(x,y)=e^{-tP_2}(x,y)\,\text{ for }(t,x,y)\in (0,\infty)\times U\times U.
\end{equation}
\end{prop}
\begin{proof}   For integer $k>0$, we have  $P^kf:=P_1^kf=P_2^kf\in C_0^\infty(U;E)$ for $f\in C_0^\infty(U;E)$, hence  $\mathcal{L}^{\operatorname{frac}}_1P^kf(x)=\mathcal{L}^{\operatorname{frac}}_2P^kf(x)$ for $x\in \Omega_1$. Consider open sets $\Omega_1,\Omega_2\Subset U$ whose closures are disjoint, and choose $f\in C_0^\infty(\Omega_1;E)$. Expressing the fractional power by \eqref{eq_gamma}, we find for $x\in U$

\begin{equation}\label{eq_idgamma1}
\int_0^\infty t^{s-1}(e^{-tP_1}-e^{-tP_2})P^kf(x)\,dt=0.
\end{equation}

Since $e^{-tP_i}$ is generated by $P_i$ and $e^{-tP_i}f$ is smooth, we have 
\begin{equation}\label{eq_commute}
    e^{-tP_i}P^kf=P^ke^{-tP_i}f=(-\partial_t)^ke^{-tP_i}f,
\end{equation}
on $U$, and \eqref{eq_idgamma1} implies
\begin{equation}\label{eq_idgamma2}
\int_0^\infty t^{s-1}\partial_t^k(e^{-tP_1}-e^{-tP_2})f(x)\,dt=0.
\end{equation}

Consider any point $x\in \Omega_2$. We proceed to integrate \eqref{eq_idgamma2} by parts $k$ times to find 
\begin{equation}
    \int_0^\infty t^{-(k+1)}\frac{(e^{-tP_1}-e^{-tP_2})f(x)}{t^{-s}}\,dt =0.
\end{equation}
This follows because the boundary terms in the integration by parts disappear, as seen from using the heat kernel to write
\begin{equation}
(-\partial_t)^k(e^{-tP_1}-e^{-tP_2})f(x)=\int_{\Omega_1} (e^{-tP_1}(x,y)-e^{-tP_2}(x,y))P^kf(y)\,dV(y).
\end{equation}
and observing the decay provided by Lemma \ref{lem_kerbounds}: decay at zero is given by \eqref{eq_ptbd} and $d_g(\Omega_1,\Omega_2)>0$, and decay at infinity is given by \eqref{eq_opbd}.

Substituting $t\mapsto 1/t$ in the above and re-indexing $k$ shows that 
\begin{equation}\label{eq_fourier}
\int_0^\infty t^k\frac{(e^{-\frac{1}{t}P_1}-e^{-\frac{1}{t}P_2})f(x)}{t^{s}}\,dt=0
\end{equation}
on $\Omega_2$ for all $k\geq 0$. Let \[\varphi(t):=\chi_{(0,\infty)}\frac{(e^{-\frac{1}{t}P_1}-e^{-\frac{1}{t}P_2})f(x)}{t^{s}}.
\]
The exponential bound in \eqref{eq_ptbd} implies that the Fourier transform $\widehat{\varphi}(\xi)$ is holomorphic for $\operatorname{Im}(\xi)<c$, $c>0$, and \eqref{eq_fourier} implies that $\widehat{\varphi}(\xi)$ vanishes to infinite order at $\xi =0$. We conclude that $(e^{-tP_1}-e^{-tP_2})f(x)=0$ for $t>0$ and $x\in\Omega_2$. In fact, by unique continuation of the heat equation \eqref{eq_heat}, we have  $(e^{-tP_1}-e^{-tP_2})f(x)\equiv 0$ on all of $(0,\infty)\times U$ (see \cite{lin1990uniqueness}). Since this holds for any $\Omega_1$ and $f\in C_0^\infty(\Omega_1; E)$, we see that $e^{-tP_1}(x,y)=e^{-tP_2}(x,y)$ on $(0,\infty)\times U$.
\end{proof}

We adapt this proposition to the setting of Theorem \ref{thm_main} by pulling back $P_1$ and the kernel $e^{-tP_1}(x,y)$ locally on $U_2$ to find:

\begin{prop}\label{prop_heatkercor}
Let $\mathcal M_1,\mathcal M_2$ satisfy the assumptions of Theorem \ref{thm_main} with local structure-preserving isomorphism $\tilde \Psi$ as in \eqref{fig_commute1}. If  $\tilde\Psi^*\mathcal{L}_1^{\operatorname{frac}}=\mathcal{L}_2^{\operatorname{frac}}\tilde\Psi^*$ on $C_0^\infty(U_1;E_1)\cap \operatorname{Ker}(P_1)^\perp$, then we have the local equality of heat kernels 
\begin{equation}\label{eq_heatloc}
    \tilde\Psi^*e^{-tP_1}(x,y)=e^{-tP_2}(x,y)\,\text{ for }(t,x,y)\in (0,\infty)\times U_2\times U_2.
\end{equation}
\end{prop}

\subsection{The wave source-to-solution map}
Having shown that $\mathcal{L}^{\operatorname{frac}}_{P,U}$ determines the heat kernel of $P$ locally, we now show that knowledge of the heat kernel gives the source-to-solution map for a wave equation. Namely, consider
\begin{align}\label{eq_wave}
\begin{cases}
(\partial_t^2+P)w^f=f \, \text{ on }(0,\infty)\times M,\\
w^f|_{\{t=0\}}=\partial_tw^f|_{ \{t=0\}}=0.\\
\end{cases}
\end{align}
For smooth sections $f\in C^{\infty}((0,\infty)\times M;E)$, the unique solution to \eqref{eq_wave} is given by \begin{equation}\label{eq_wavesoln}
    w^f(t,x)=\int_0^t G(t-s,P)f(s,x)\,ds,
\end{equation}
where $G(s,P)$ is the wave kernel defined via the functional calculus with 
\begin{equation}
    G(s,\lambda):=\sum_{k=0}^\infty \frac{s^{2k+1}\lambda^k}{(2k+1)!}=
    \begin{cases}
\sin{(s\sqrt{\lambda})}/\sqrt{\lambda}\text{ for }\lambda>0,\\
\sinh{(s\sqrt{-\lambda})}/\sqrt{-\lambda}\text{ for }\lambda<0.
\end{cases}
\end{equation}

\begin{definition}\label{def_waveDN}
The local wave source-to-solution map $\mathcal{L}_{P,U}^{wave}:C_0^\infty((0,\infty)\times U; E)\to C^\infty((0,\infty)\times U;E)$ is given by $\mathcal{L}_{P,U}^{wave}f:=w^f|_{(0,\infty)\times U}$, with $w^f$ as in \eqref{eq_wavesoln}.
\end{definition}

The link between $\mathcal{L}^{\operatorname{frac}}_{P,U}$ and $\mathcal{L}_{P,U}^{wave}$ is given by Proposition \ref{prop_heatker} and the following:
\begin{prop}\label{prop_transmute}
Let $E\to M$, $(g_i,\g{\cdot}{\cdot}_i, \nabla_i,A_i)$ be as in Proposition \ref{prop_heatker}, and let $\mathcal{L}^{wave}_i:=\mathcal{L}_{P_i,U}^{wave}$ for $U\subset M$.  Equality of the heat kernels $e^{-tP_1}(x,y)=e^{-tP_2}(x,y)$ on $(0,\infty)\times U\times U$  implies that  $\mathcal{L}^{wave}_1=\mathcal{L}^{wave}_2$.
\end{prop}

\begin{proof}
By \eqref{eq_wavesoln}, it suffices to show that $G(s,P_1)u(x)=G(s,P_2)u(x)$ for all $u\in C_0^\infty(U;E)$, $x\in U$. To do so, we exploit a version of the Kannai transmutation formula as used in \cite{feiz2021}.
By Theorem 2.1 in \cite{ludewig2019strong}, we have 
\[e^{-tP}u=\frac{1}{(4\pi t)^{1/2}}\int_{-\infty}^\infty e^{-s^2/4t}G'(s,P)\,ds\]
Because $e^{-s^2}G'(s,\lambda)\to 0$ as $s\to\pm\infty$, we integrate by parts to find
\[e^{-tP}u=\frac{1}{4\pi^{1/2}t^{3/2}}\int_{0}^\infty e^{-s/4t}G(s,P)u\,ds.\]

For smooth sections $u\in C_0^\infty(U;E)$, we have $e^{-tP_1}u=e^{-tP_2}u$ on $U$ by assumption. Therefore, using the Laplace transform $\mathcal L$, we find on $U$:
\[\mathcal{L}(G(\cdot,P_1)u)(1/4t)=\int_{0}^\infty e^{-s/4t}G(s,P_1)u\,ds=\int_{0}^\infty e^{-s/4t}G(s,P_2)u\,ds=\mathcal{L}(G(\cdot,P_2)u)(1/4t).\]
Taking the inverse Laplace transform gives $G(s,P_1)u(x)=G(s,P_2)u(x)$ for $s>0$, $x\in U$.
\end{proof}
Once again, computing with the local pullback $\tilde\Psi^*$ yields the following:                        
\begin{prop}\label{prop_transmutecor}
Let $\mathcal M_1$, $\mathcal M_2$ satisfy the assumptions of Theorem \ref{thm_main} with local structure-preserving isomorphism $\tilde \Psi$ as in \eqref{fig_commute1}. Consider $\mathcal{L}^{wave}_i:=\mathcal{L}_{P_i,U}^{wave}$ for $i=1,2$. If the local equality of heat kernels \eqref{eq_heatloc} holds, then $\tilde\Psi^*\mathcal{L}^{wave}_1=\mathcal{L}^{wave}_2\tilde\Psi^*$.
\end{prop}

\section{Wave propagation and cut times}

We have now reduced our fractional inverse problem to one involving a linear wave equation and its associated  $\mathcal{L}^{wave}_{P,U}$ for which we have techniques used in  \cite{helin2018correlation},\cite{kurylev2018inverse}. We proceed to recount these and other necessary tools in this section.

\subsection{Unique continuation and approximate controllability}

The results in this subsection hold for all complete manifolds $M$ and are the main tools we use to recover geometric structure from the Dirichlet-to-Neumann data. 

The first fact we will use carries over from the scalar setting:
\begin{prop}\label{prop_speed}
(Finite speed of propagation) Let $f\in L^2(\RR\times M;E)$, and suppose $u$ is a solution to
\begin{align}
\begin{cases}
(\partial_t^2+P)u=f \, \text{ on }(0,\infty)\times M,\\
f|_{C(T,p)}=0,\\
u|_{B(p,T)\times \{t=0\}}=\partial_tu|_{B(p,T)\times \{t=0\}}=0.
\end{cases}
\end{align}
Then $u\equiv 0$ on $C(T,p)$.
\end{prop}
Consider the ``diamond set'' $C(T,U):=\{(t,p)\in(0,2T)\times M: d_g(p,U)<\min(t,2T-t)\}$ containing $U\subset M$. We also denote $M(T,U):=\{p\in M: d_g(p,U)\leq T\}$.
Because $P$ is principally scalar, the following is due to the Carleman estimates of \cite{eller2002uniqueness} and the method of proof outlined in \cite{kachalov2001inverse}, Section 2.5:

\begin{prop}\label{prop_ucont}
(Unique continuation) Suppose $U\subset M$ is a bounded open set. If $u\in C_0^\infty(\RR\times M;E)$ solves $(\partial_t^2+P)u=0$ in $(0,2T)\times M(T,U)$ and $u|_{(0,2T)\times U}\equiv 0$, then $u\equiv 0$ on $C(T,U)$.
\end{prop}
These two propositions above combine to yield two important lemmas.
\begin{lem}\label{prop_approxcont}
(Approximate controllability) Let $U\subset M$ be a bounded open set. For any $T>\eps>0$, the set of solutions $\mathcal W_T:=\{w^f(T,\cdot):f\in C_0^\infty((T-\eps,T)\times U;E),\,w^f\text{ solves \eqref{eq_wave}}\}$ is dense in $L^2(M(\eps,U); E)$.
\end{lem}

\begin{proof}
We show that $\mathcal W_T^\perp=\{0\}$ in $L^2(M(\eps,U);E)$, from which the claim follows. Let $\phi\in \mathcal W_T^\perp\subset L^2(M(\eps,U);E)$ and $u\in C^\infty((0,\infty)\times M)$ be the solution in $(0,T)\times M$ to 
\begin{align}
\begin{cases}
(\partial_t^2+P)u=0\\\label{eq_homwave}
u|_{t=T}=0, \partial_tu|_{t=T}=\phi.
\end{cases}
\end{align}
For any $\psi\in \mathcal C_0^\infty((T-\eps,T)\times U;E)$, we have 
\[(u,\psi)_{L^2((0,T)\times M;E)}=(u,(\partial_t^2+P)w^\psi)_{L^2((0,T)\times M;E)}=((\partial_t^2+P)u,w^\psi)_{L^2((0,T)\times M;E)}=0,\]
where the boundary terms at $t=0,T$ vanish in part due to $(\phi,w^\psi(T,\cdot))_{L^2}=0$ by assumption. (In the case that $M$ is not closed but merely complete, we further use Proposition \ref{prop_speed} in the above calculation.) By density, we conclude $u=0$ on $(T-\eps,T]\times U$.

Using the odd reflection $-u(2T-t,x)$, we also construct another solution to \eqref{eq_homwave} on $(T,2T)\times M$ with the same initial conditions at $t=T$. These two solutions combine to give a solution $\tilde u$ to \eqref{eq_homwave} on $(0,2T)\times M$ satisfying $\tilde u=0$ on $(T-\eps,T+\eps)\times U$. After translating in time, Proposition \ref{prop_ucont} implies $\tilde u= 0$ on 
\[\{(t,p)\in(T-\eps,T+\eps)\times M: d_g(p,U)<\min(t-(T-\eps),(T+\eps)-t)\}\]
and in particular $\partial_t\tilde u|_{t=T}=\phi|_{M(\eps,U)}=0$ as needed.
\end{proof}

We now define  $\mathcal F(T,U):=\{f\in C_0^\infty(\RR\times M;E):supp(f)\subset (0,T)\times U\}$ and consider the time averaging operator $J\phi(t):=\frac{1}{2}\int_t^{2T-t} \phi(s)\,ds$.
\begin{lem}\label{prop_blag}
(Blagovestchenskii identity) Let $U\subset M$ be a bounded open set and $T>0$. For any $f,h\in\mathcal{F}(2T,U)$, we have:
\[\rg{w^f(T,\cdot)}{w^h(T,\cdot)}_{L^2}=\rg{f}{(J\mathcal{L}_{P,U}^{wave}-{\mathcal{L}_{P,U}^{wave}}^*_UJ)h}_{L^2}\]
\end{lem}
\begin{proof}
Observe that $\rg{\partial_t^2w^f(t,\cdot)}{w^h(s,\cdot)}_{L^2}=\rg{f(t,\cdot)}{\mathcal{L}_{P,U}^{wave}h(s,\cdot)}_{L^2}-\rg{Pw^f(t,\cdot)}{w^h(s,\cdot)}_{L^2}$ because \[supp(f(t,\cdot))\subset U.\] Consider $W\in C^\infty((0,\infty)^2; \RR)$ given by $W(t,s):=\rg{w^f(t,\cdot)}{w^h(s,\cdot)}_{L^2}$ and $F(t,s):=(\partial_t^2-\partial_s^2)W(t,s)$. By our above observation, we find that $W$ solves the wave equation 
\begin{align*}
\begin{cases}
(\partial_t^2-\partial_s^2)W=F \, \text{ on }(0,2T)\times(0,\infty),\\
W|_{\{t=0\}}=\partial_tW|_{ \{t=0\}}=0,\\
\end{cases}
\end{align*}
hence is given by 
\[W(t,s)=\frac{1}{2}\int_0^t\int_\tau^{2T-\tau}F(\tau, y)\,dyd\tau.\]
Changing variables yields
$W(T,T)=\rg{f}{J\mathcal{L}_{P,U}^{wave}h}_{L^2}-\rg{\mathcal{L}_{P,U}^{wave}f}{Jh}_{L^2}$ as needed.
\end{proof}

\subsection{Cut times and geodesically transported neighborhoods}
We review some basic properties of cut times which will be of use in metric reconstruction, and we further construct special neighborhoods that are adapted for our proof of bundle reconstruction. In this subsection, $(M,g)$ is a complete, connected, smooth Riemannian manifold.

Recall that the cut time $t^*:SM\to (0,\infty]$ is a continuous function on the unit tangent bundle given by $t^*(x,v):=\inf \{t>0 :\gamma_{x,v}([0,t])\text{ is length minimizing}\}$, where $\gamma_{x,v}$ is the unit speed geodesic with initial point and direction $(x,v)$. The following lemma finds $t^*(x,v)$ in terms of containment of balls:

\begin{lem}\label{lem_cut}
Let $x,y\in M$ and $s:=d_g(x,y)>0$. Suppose $\gamma_{x,v}$ is a geodesic starting at $x$ whose restriction to the interval $[0,s]$ is the length minimizer joining $x$ to $y$. Let $\mathcal I(x,y)$ be the collection of times where the following containment is possible:
\begin{align*}
    \mathcal I(x,y):= \{r>s:\text{ there exists } \veps(r)>0\text{ such that } B(y,r-s+\veps(r)) \subset B(x,r)\}.
\end{align*}
Then the cut time $t^*(x,v)=\inf \mathcal I(x,y)$
\end{lem}
\begin{proof}
Let $r\in \mathcal{I}(x,y)$ and $z:=\gamma_{x,v}(r+\delta)$ for any $\delta<\veps(r)$. Since $d_g(y,z)\leq\ell(\gamma_{x,v}([s,r+\delta])$, we find $z\in B(y, r-s+\veps(r))\subset B(x,r)$, hence $\gamma_{x,v}$ is not length minimizing from $x$ to $z$. Therefore, $t^*(x,v)<r+\delta$, and we conclude $t^*(x,v)\leq \inf \mathcal I(x,y)$. Conversely, let $r'=t^*(x,v)+\delta$. In order to show $r'\in\mathcal{I} (x,y)$, it suffices to show $\partial B(y,r'-s)\subset B(x,r')$. Let $z'\in \partial B(y,r'-s)$, so $d_g(x,z')\leq r'$. Observe that if equality is achieved, then $z'=\gamma_{x,v}(r')$ because any other broken geodesic passing through $y$ can be shortened by smoothing. However, this contradicts $r'> t^*(x,v)$, so we must have $d_g(x,z')<r'$, thereby concluding the proof.
\end{proof}

The cut locus of $x$ is given by $\operatorname{Cut}(x):=\{q=\gamma_{x,v}(t^*(x,v))$, and the injectivity domain of the exponential map at $x$ is given by $\mathcal{D}^{\mathrm{inj}}(x):=\{sv\in T_xM:s<t^*(x,v)\}\subset T_xM$. A proof of the following can be found in \cite{lee2018introduction}:
\begin{lem}\label{lem_cutprop}
For any $x\in M$, we have the following:
\begin{enumerate}[(a)]
    \item $\operatorname{Cut}(x)\subset M$ has measure zero. 
    \item The restricted exponential map $\exp_x:\mathcal{D}^{\mathrm{inj}}(x)\to M\setminus \operatorname{Cut}(x)$ is a diffeomorphism.
\end{enumerate}
\end{lem}

For points within $\operatorname{Cut}(x)$, we construct the following pairs of open sets that shrink to points as $k\to\infty$ by transporting polar neighborhoods of $x$ along distance minimizing geodesic segments. 
\begin{lem}
Suppose $y=\gamma_{x,v}(s)$ for some $0<s<t^*(x,v)$, and let $\pi:E\to M$ be a smooth vector bundle. For sufficiently large $k>k^*>0$, there are neighborhoods $(X_k)_{k\geq k^*}, (Y_k)_{k\geq k^*}$ of $x,y$ respectively such that
\begin{enumerate}
    \item $\lim_{k\to\infty}\operatorname{Diam}(X_k)=0=\lim_{k\to\infty}\operatorname{Diam}(Y_k)$;
    \item $d_g(X_k,Y_k)<s+1/k$;
    \item Both $E|_{X_k}$ and $E|_{Y_k}$ are trivial.
\end{enumerate}
\end{lem}

\begin{proof}
By extending the geodesic segment $\gamma_{x,v}([0,s])$ past $x$, we can find $(x',v')\in SM$ such that $x=\gamma_{x',v'}(r')$ and $y=\gamma_{x',v'}(s+r')$ with $s+r'<t^*(x',v')=t^*(x,v)$. Taking $r'>0$ small enough so that $x$ is contained within a polar normal chart $(\mathcal O',(r,\theta^j))$ centered at $x'$, we see $x=(r',\theta^j)$ in these polar normal coordinates. By continuity of $t^*$, we see that for small enough $\eps,\delta>0$,
\begin{equation}\label{eq_shrinkx}
    \mathcal U(x,\eps,\delta):=(r'-\eps,r'+\eps)\times(\theta^j-\delta,\theta^j+\delta)\subset \mathcal O'
\end{equation}
is a neighborhood of $x$ such that
\begin{equation}\label{eq_shrinky}
    \mathcal V(y,\eps,\delta):=\exp_{x'}(s\mathcal U(x,\eps,\delta))\subset M\setminus \operatorname{Cut}(x'),
\end{equation}
where we regarded $\mathcal U(x,\eps,\delta)$ as a subset of $T_{x'}M$ in \eqref{eq_shrinky}. Observe that with this construction, we can also turn $\mathcal V(y,\eps,\delta)$ into a coordinate chart by using the same polar coordinates associated with $\mathcal O'$ on $\mathcal U(x,\eps,\delta)$ and translating the radial coordinate by $s$. This corresponds to how every point in $\mathcal V(y,\eps,\delta)$ is joined to a unique point in $\mathcal U(x,\eps,\delta)$ via a length $s$ segment of a radial geodesic emanating from $x'$. Moreover, both \eqref{eq_shrinkx} and \eqref{eq_shrinky} are contractible, hence any vector bundle over those respective sets are trivial.

Choosing $k>0$ large enough such that $1/k=\eps=\delta$ satisfies the above allows us to construct $X_k$ using \eqref{eq_shrinkx} and $Y_k$ using $\eqref{eq_shrinky}$.
\end{proof}

\section{Proof of Theorem}
\subsection{Recovery of the metric}
We adapt the method of \cite{helin2018correlation} to our bundle-valued setting to determine the metric $g$. Let $S(x,\veps,\tau):=(T-\tau,T)\times B(x,\veps)$. Due to  finite speed of propagation and approximate controllability of the wave equation, we have the following characterization of containment of balls using solutions to \eqref{eq_wave}:
\begin{lem}\label{lem_balls}
For $x,y,z\in M$ and $\tau_x,\tau_y,\tau_z>\delta>0$, the following conditions are equivalent:
\leavevmode
\begin{enumerate}
\item $B(x,\tau_x)\subsetneq {B(y,\tau_y)\cup B(z,\tau_z)}$
\item For all $f\in C_0^\infty(S(x,\delta, \tau_x-\delta);E)$, there is a sequence $(f_j)_{j=1}^\infty\subset C_0^\infty(S(y,\delta, \tau_y-\delta)\cup S(z,\delta, \tau_z-\delta);E)$ whose respective solutions to \eqref{eq_wave} approximate the solution for $f$, that is:
\[\lim_{j\to\infty}\|w^f(T,\cdot)-w^{f_j}(T,\cdot)\|_{L^2(M;E)}= 0\]
\end{enumerate}
\end{lem}
\begin{proof}
For the forward implication, choose $f\in C_0^\infty(S(x,\delta, \tau_x-\delta);E)$, so in particular the restriction to backward cones $f|_{C(T,p)}=0$ for all points satisfying $d_g(p,B(x,\delta))\geq \tau_x-\delta$. Lemma \ref{prop_speed} then yields $supp(w^f(T,\cdot))\subset B(x,\tau_x)\subset {B(y,\tau_y)\cup B(z,\tau_z)}$, the latter containment following by assumption. Let $\chi_y$ be the characteristic function for $B(y,\tau_y)$, and rewrite $w^f(T,\cdot)=\chi_y w^f(T,\cdot)+(1-\chi_y)w^f(T,\cdot)$. We see $\chi_y w^f(T,\cdot)\in L^2(B(y,\tau_y);E)=L^2(M(\tau_y-\delta,B(y,\delta));E)$, hence Lemma \ref{prop_approxcont} implies $\chi_y w^f(T,\cdot)$ can be approximated in $L^2(B(y,\tau_y);E)$ by $w^{\phi^{(y)}_j}(T,\cdot)$ for $\phi^{(y)}_j\in S(y,\delta, \tau_y-\delta)$. Similarly, one can find a suitable sequence approximating $(1-\chi_y)w^f(T,\cdot)$, and the sequence constructed from their sum gives the result.

Conversely, suppose (1) does not hold and take a nonempty open set $U\subset B(x,\tau_x)\setminus {B(y,\tau_y)\cup B(z,\tau_z)}$. By Lemma \ref{prop_speed}, $w^{\tilde f}(T,\cdot)=0$ on $U$  for all $\tilde{f}\in C_0^\infty(S(y,\delta,\tau_y-\delta)\cup S(z,\delta,\tau_z-\delta);E)$, whereas we can find $f\in C_0^\infty(S(x,\delta,\tau_x-\delta);E)$ such that $\|w^f(T,\cdot)\|_{L^2}>0$ by Lemma \ref{prop_approxcont}. Therefore, (2) also does not hold.

\end{proof}

Our approach requires the restricted distance function $d_g|_{U\times U}$. Observe that while we have knowledge of $(U,g|_U)$, we do not assume $U\subset M$ is geodesically convex, nor do we initially have the information to shrink to such an appropriate subset of $U$. Instead, we use properties of the wave equation to prove the following:

\begin{prop}
  The topological structure of $U$ and the wave source-to-solution data $\mathcal{L}_{P,U}^{wave}$ determine the local distance map $d_g(\cdot,\cdot)|_{U\times U}$.
\end{prop}
\begin{proof}
The topology of $U\subset(M,g)$ is metrizable, so we may fix a compatible metric on $U$ and consider the metric space $(U,d_0)$. Let $B_0(p,r)$ denote the open metric balls with respect to $d_0$. For $x,y\in U$ and $\eps>0$, consider the sets 
\[D(x,y,k):=\{t>0:\text{ there exists } f\in C_0^\infty((0,\infty)\times B_0(x,1/k))\text{ such that } supp(w^f(t,\cdot))\cap B_0(y,1/k)\neq \varnothing\}\]
that can be determined by $\mathcal{L}_{P,U}^{wave}$. 
Using Lemma \ref{prop_approxcont} and similar reasoning as in the proof of Lemma \ref{lem_balls}, we find  $d_g(x,y)=\liminf_{k\to\infty}D(x,y,k)$.
\end{proof}

\begin{prop}\label{prop_rDist}
  Let $U\subset M$ be a nonempty open set. The restricted Riemannian structure $(U,g|_U,d_g|_{U\times U})$ and the wave source-to-solution data $\mathcal{L}_{P,U}^{wave}$ determines on $U$ the family of distance functions \[\mathcal{R}_U(M):=\{d_g(p,\cdot)|_{U}:p\in M\}.\]
\end{prop}
\begin{proof}
We first note that because $U$ is open and nonempty, any point in $M$ can be reached by a unit speed geodesic emanating from $U$ before its cut time. Indeed, if this were not true, $U$ would be contained in the cut locus of a point, which has measure zero, leading to a contradiction. Therefore, it suffices to show that we can determine $t^*$ on $SM|_U$ and, further, find $d(p,\cdot)|_U$ for any $p=\gamma_{x,v}(r')$ with $(x,v)\in SM|_U$, $r'<t^*(x,v)$.

To this end, for any $(x,v)\in SM|_U$ we choose $s>0$ small enough such that $\gamma_{x,v}([0,s])\subset U$. Letting $y:=\gamma_{x,v}(s)$, we see $\mathcal{L}_{P,U}^{wave}$ determines $\mathcal I(x,y)$ via Lemma \ref{prop_blag}, polarization,  and Condition 2 of Lemma \ref{lem_balls}. Therefore, we know $t^*(x,v)$ by the characterization of Lemma \ref{lem_cut}.

For $z\in U$, we consider the set 
\begin{align*}
    \mathcal I(x,y,z):= \{r>0:\text{ there exists } \veps(r)>0\text{ such that } B(y,r'-s+\veps(r)) \subset B(x,r')\cup B(z,r)\}.
\end{align*}
 We claim that $d(p,z)=\inf \mathcal I(x,y,z)$. Indeed, $d(p,x)=r'$, so for $r\in\mathcal I(x,y,z)$, we have $p\in B(z,r)$ since $p\in B(y,r'-s+\veps(r))$; thus, $d(p,z)\leq \inf\mathcal I(x,y,z)$. Now, let $\tilde r:=d(p,z)+\delta$ for some $\delta>0$. If $\tilde r\notin \mathcal I(x,y,z)$, then by passing to a subsequence we may find points
\[p_k\in B(y,r'-s+1/k)\setminus (B(x,r')\cup B(z,\tilde r))\subset \overline{B(y,r'-s+1)}\]
such that $p_k\to p'\in \partial B(y,r'-s)$. Observe that $d(x,p')\leq r'$, with equality achieved if and only if $p=p'$ by a smoothing argument as in the proof of Lemma \ref{lem_cut}. The equality $p=p'$ and the fact $\tilde r>d(p,z)$ would contradict the fact that our limiting sequence stays outside $B(z,\tilde r)$; however, $d(x,p')<r'$ would imply our limiting sequence does not stay outside $B(x,r')$. Therefore, we find $\inf \mathcal I(x,y,z)\leq d(p,z)$ and our claim follows.

Because $\mathcal I(x,y,z)$ can be determined from our initial data by similar methods as those used for $\mathcal I(x,y)$ earlier, we are able to deduce $d(p,z)$ and ultimately determine $\mathcal R_U(M)$.
\end{proof}

\begin{prop}\label{prop_rDistcor}
Let $\mathcal M_1$, $\mathcal M_2$ satisfy the assumptions of Theorem \ref{thm_main} with local structure-preserving isomorphism $\tilde \Psi$ as in \eqref{fig_commute1}. Choose any $\mathcal O_2\Subset U_2$. If $\tilde\Psi^*\mathcal{L}^{wave}_1=\mathcal{L}^{wave}_2\tilde\Psi^*$, then there is an isometry $\psi: M_2\to M_1$ with $\psi|_{\mathcal O_2}=\tilde \psi|_{\mathcal O_2}$.

\end{prop}
\begin{proof}
Choosing $\mathcal O_2=\tilde\psi(\mathcal O_1)$ as necessary, we are able to determine 
\begin{equation}
    \mathcal R_i(M_i):=\{d_{g_i}(p,\cdot)|_{\overline{\mathcal O_i}}:p\in M_i\}\,\text{ for }i=1,2
\end{equation}
by our previous proposition. As shown in \cite{helin2018correlation}, the map 
\[R_i: M_i\ni p\mapsto d_{g_i}(p,\cdot)|_{\overline{U}_i}\in \mathcal R_i(M)\] is a diffeomorphism when $\mathcal R_i(M)$ is equipped with a smooth structure constructed explicitly using the injectivity domain of points in $\overline U_i$, which is known from $\mathcal{L}^{wave}_i$ using Proposition \ref{prop_rDist}. Equipping $\mathcal R_i(M)$ with the pullback of $g_i$ using $\inv R_i$ makes $R_i$ an isometry, and by Proposition 5 in \cite{helin2018correlation} the coordinate representation of this metric on $\mathcal R_i(M)$ is determined by Proposition \ref{prop_rDist}. Because $\tilde\Psi^*\mathcal{L}^{wave}_1=\mathcal{L}^{wave}_2\tilde\Psi^*$, the map $\psi_R:\mathcal R_1(M_1)\to \mathcal R_2(M_2)$ defined using the pullback of distance functions by $\tilde \psi$ is an isometry. Setting $\psi:=\inv R_1\inv\psi_RR_2$ yields a global isometry with the required properties.
\end{proof}
By Propositions \ref{prop_heatkercor}, \ref{prop_transmutecor}, and \ref{prop_rDistcor}, the existence of $\psi$ in Theorem \ref{thm_main} follows.
\subsection{Recovery of the bundle and operator}

Recovery of our remaining geometric quantities follows the procedure of \cite{kurylev2018inverse} in the setting of closed manifolds.
\begin{lem}\label{lem_supp}
Let $U_1,U_2\subset M$ be open sets and $\eps_1,\eps_2>0$. Suppose there exists a sequence $(f_j)_{k=1}^\infty\subset C_0^\infty((T-\eps_1,T)\times U_1)$ such that:
\begin{enumerate}
\item The sequence of solutions $(w^{f_j}(T,\cdot))_{j=1}^\infty$ converges weakly to $\phi\in L^2(M;E)$;
\item For all $h\in C_0^\infty((T-\eps_2,T)\times U_2)$, we have $\lim_{j\to\infty}\rg{w^{f_j}(T,\cdot)}{w^h(T,\cdot)}_{L^2}=0$.
\end{enumerate}
Then $supp(\phi)\subset M(\eps_1,U_1)\setminus M(\eps_2,U_2)^{int}$.
\end{lem}
\begin{proof}
This follows from the same argument used in the proof of Lemma \ref{lem_balls}.
\end{proof}

\begin{lem}\label{lem_frame}
Consider an open set $U\subset M$. For any $T>0$ and point $p\in M$ such that $d_g(p,U)<T$, there exists a tuple $(h_\ell)_{\ell=1}^{rk(E)}\subset \mathcal{F}(2T,U)$ such that $(w^{h_\ell}(T,\cdot))_{\ell=1}^{rk(E)}$ is a smooth local frame for $E$ in a neighborhood $x\in V_x\subset M$ and is orthonormal at  $x$.
\end{lem}
\begin{proof}
It suffices to show $w^h(T,x)$, $h\in\mathcal{F}(T,U)$, spans $E_x$; our claims then follow by Gram-Schmidt and smoothness of $w^h(T,\cdot)$. To this end, we show that if $e\in E_x$ is such that $\g{e}{w^h(T,x)}_E=0$, then $e=0$. If we let $\psi:=e\delta_x$ in \eqref{eq_homwave}, our assertion follows from the same reasoning as the proof of Lemma \ref{prop_approxcont}.
\end{proof}

The following two lemmas identify certain double sequences of sections with vectors in the fibre $E_y$, where $y$ may not be contained in the support of the sequence. For convenience, we define $s_k:=s+1/k$ where $\gamma_{x,v}(s)=y$, $s<t^*(x,v)$. In the case $s>0$, we construct $X_k$, $Y_k$ as in \eqref{eq_shrinkx} and \eqref{eq_shrinky}; otherwise, set $Y_k=X_k$ as defined in $\eqref{eq_shrinkx}$. Observe that in either case, $Y_k\subset M(s_k,X_k)$.

\begin{lem}\label{lem_dseq}
Let $U\subset M$ be open, and consider $y=\gamma_{x,v}(s)\in M$ for $(x,v)\in SM|_{U}$, $s<t^*(x,v)$. Suppose we have a double sequence of sections $\Phi^y:=(f_{jk})_{j,k=1}^\infty$ with $f_{jk}\in C_0^\infty((T-s_k,T)\times X_k;E)$ such that:
\begin{enumerate}
    \item For each $k$,  $(w^{f_{jk}}(T,\cdot))_{j=1}^\infty$ converges weakly in $L^2(M;E)$ to a section $u_k$ supported in $Y_k$;
    \item There is a uniform $C>0$ such that $\|w^{f_{jk}}(T,\cdot)\|_{L^2}<C(Vol(Y_k))^{-1/2}$;
    \item For any $h\in\mathcal F(2T,U)$, $\lim_{k\to\infty}\lim_{j\to\infty}\rg{w^{jk}(T,\cdot)}{w^h(T,\cdot)}_{L^2}$ exists.
\end{enumerate}
Then there is a vector $e(\Phi^y)\in E_y$ depending on $\Phi^y$ such that for all $\phi\in C^\infty(M;E)$,
\begin{equation}\label{eq_approxvec}
\lim_{k\to\infty}\lim_{j\to\infty}\rg{w^{f_{jk}}(T,\cdot)}{\phi}_{L^2}=\g{e(\Phi^y)}{\phi}_E.
\end{equation}
\end{lem}
\begin{proof}
By taking $k$ large enough, we may assume without loss of generality that $X_k\subset U$ and $Y_k$ is contained within a fixed coordinate chart $(W,\tilde y)$ with $\tilde y=0$ at $y$. Let $b_\ell(\cdot):=w^{h_\ell}(T,\cdot)$ denote the sections of the local frame for $E$ constructed using Lemma \ref{lem_frame} which are orthonormal at $E_y$. Any $\phi\in C^\infty(M;E)$ can be written locally using a linearization as 
\[\phi(\tilde y)=\sum_{\ell=1}^{rk(E)} \g{\phi}{b_\ell}_Eb_\ell=c^\ell b_\ell(\tilde y) +\tilde y^i\psi_i(\tilde y)\]
where $\tilde \psi_k\in C^\infty(W;E)$, $1\leq i\leq n$, and $c^\ell\in \CC$. Taking the inner product with the weak limit in (1) yields
\[\rg{u_k}{\phi}_{L^2}=\overline{c^\ell}\rg{u_k}{b_\ell}_{L^2}+r_k\]
where 
\begin{align*}
    |r_k|&< n\max_{i}\|\psi_i\|_{C(W)}\operatorname{Diam}(Y_k)\int_{Y_k}|u_k|\,dV\\
    &<n\max_{i}\|\psi_i\|_{C(W)}\operatorname{Diam}(Y_k)|(Vol(Y_k))^{1/2}\|u_k\|_{L^2}.
\end{align*}
Since $Y_k\to y$ by construction, (2) implies $r_k\to 0$, and (3) implies we can set $a^\ell :=\lim_{k\to\infty} \rg{u_k}{b_\ell}_{L^2}$. Letting $e_y^S:=a^\ell b_\ell(y)$, we find
\[\lim_{k\to\infty}\lim_{j\to\infty}\rg{w^{f_{jk}}(T,\cdot)}{\phi}_{L^2}=\lim_{k\to\infty}\rg{u_k}{\phi}_{L^2}=\sum_{\ell=1}^{rk(E)}\overline{c^\ell}a^\ell=\g{e(\Phi^y)}{\phi}_E.\]
\end{proof}
Conversely, we also have: 
\begin{lem}\label{lem_conv}
Let $U\subset M$ be open, and consider $y=\gamma_{x,v}(s)\in M$ for $(x,v)\in SM|_{U}$, $s<t^*(x,v)$. For any $e_y\in E_y$, there is a double sequence $\Phi^y=(f_{jk})_{j,k=1}^\infty$ with $f_{jk}\in C_0^\infty((T-s_k,T)\times X_k;E)$ satisfying conditions (1)--(3) of Lemma \ref{lem_dseq}, with $e(\Phi^y)=e_y$.
\end{lem}
\begin{proof}
For any smooth section $\tilde e\in C^\infty(M;E)$ with $\tilde e(y)=e_y$, we construct 
\[u_k\in Vol(Y_k)^{-1}\chi_k\tilde e \in L^2(M(s_k,X_k);E)\]
with $\chi_k$ the indicator function for $Y_k$. Using Lemma \ref{prop_approxcont} for each $k$, we find our double sequence $\Phi^y$ to approximate each $u_k$ strongly, and Condition (1) follows. Condition (2) is also immediate  with $C=\|\tilde e\|_{L^\infty}$, and Condition (3) and \eqref{eq_approxvec} follow by direct calculation.
\end{proof}
\begin{prop}\label{prop_rConn}
 Let $\mathcal M$ be as in \eqref{eq_structure}. If we know the Riemannian structure $(M,g)$, the local Hermitian bundle $(E,\g{\cdot}{\cdot}_E)|_U$, and $\mathcal{L}_{P,U}^{wave}$, then we can determine the structures $(E,\g{\cdot}{\cdot}_E,\nabla,A)$ globally on $M$.
\end{prop}
\begin{proof}
Shrinking $U$ as necessary, we assume $U$ is a coordinate ball that trivializes $E$. Observe that due to compactness of $M$ and Lemma \ref{lem_cutprop}, we can cover $E$ with finitely many charts of the form 
\[\mathcal A=({}^{(0)}\mathcal A,{}^{(1)}\mathcal A,\ldots, {}^{(N)}\mathcal A):=(U,{}^{(1)}\mathcal{V}, \ldots, {}^{(N)}\mathcal{V}),\]
with ${}^{(j)}\mathcal V$ constructed using \eqref{eq_shrinky}. Therefore, $E$ is trivial over each ${}^{(j)}\mathcal A$. For each $y\in ^{(j)}\mathcal A$, we can use $\mathcal{L}_{P,U}^{wave}$ to detect sequences $(\Phi^y_\alpha)_\alpha$ satisfying the conditions of Lemma \ref{lem_dseq} by Lemmas \ref{prop_approxcont}, \ref{prop_blag}, and \ref{lem_supp}. We are able to test for smoothness of $y\mapsto e(\Phi^{y})$ by checking the smoothness of $y\mapsto \g{e(\Phi^y)}{w^{h}(T,\cdot)}_E$ for $h\in \mathcal F(2T,U)$, per Condition 3 of Lemma \ref{lem_dseq} and Lemma \ref{lem_frame}. Further, due to the existence of an orthonormal frame $(e_\ell)_{\ell=1}^{rk(E)}$ over ${}^{(j)}\mathcal A$ and Lemma \ref{lem_conv}, we can select for sequences $(\Phi^y_\ell)_{\ell=1}^{rk(E)}$ such that $e_\ell(y)=e(\Phi^y_\ell)$ by verifying
\begin{equation}\label{eq_localframe}
\lim_{k\to\infty}\lim_{j\to\infty}\rg{\chi_{{}^{(j)}\mathcal A}e_\kappa(\Phi^y)}{w^{f^y_{jk,\ell}}(T,\cdot)}_{L^2}=\g{e_\kappa(y)}{e_\ell(y)}_E=\delta_{\kappa\ell}.
\end{equation}
The Hermitian metric $\g{\cdot}{\cdot}_E$ on this trivialization over ${}^{(j)}\mathcal A$ is then $\g{v}{w}_E=\delta_{\kappa\ell}a^{\kappa}\overline{c^\ell}$ for $v=a^{\kappa}e_\kappa$, $w=c^{\ell}e_\ell$. Given a similar orthonormal frame $(\tilde e_\ell)_{\ell=1}^{rk(E)}$ and associated sequence over ${}^{(k)}\mathcal A$, we apply the method of \eqref{eq_localframe} to calculate $\g{e_\kappa}{\tilde e_\ell}_E$ to find the transition functions $\tau_{jk}:{}^{(j)}\mathcal A\cap {}^{(k)}\mathcal A\to GL(rk(E),\CC)$.

Having determined the Hermitian bundle structure $(E,\g{\cdot}{\cdot}_E)\to M$, we now show how we can recover the connection $\nabla$ and potential $A$. From our reasoning above, we know the components of $w^h(T,\cdot)$ for $h\in \mathcal F(2T,U)$ on ${}^{(j)}\mathcal A$, and by a time translation we know the components of $w^h(t,\cdot)$ in the same chart, for $t\in (0,T)$ and $h\in \mathcal F(T,U)$. We therefore know the left hand side of
\begin{equation}\label{eq_connect}
   -\rg{\partial_t^2w^h(T,\cdot)}{\phi}_{L^2}=\rg{Pw^h(T,\cdot)}{\phi}_{L^2}=\rg{w^h(T,\cdot)}{P\phi}_{L^2}.
\end{equation}
for any $\phi\in C_0^\infty({}^{(j)}\mathcal A;E)$, $h\in C_0^\infty((0,T)\times U;E)$. By density of $w^h(T,\cdot)$ in $L^2({}^{(j)}\mathcal A;E)$, we deduce $P\phi$ from \eqref{eq_connect}.

Since $\nabla=d+B$ for skew-Hermitian $B\in C^\infty(M;E\otimes E^*\otimes T^*M)$, we have:
\begin{equation}\label{eq_connect1}
(P-A)\phi=\nabla^*\nabla \phi= d^*d\phi-2\tr_g (B\otimes d\phi)+(d^*B)\phi-\tr_g (B\otimes B\phi)
\end{equation}
where the contractions are between the $E^*\otimes T^*M$ and $E\otimes T^*M$ entries.
Locally on ${}^{(j)}\mathcal A$, if we choose $\phi=\phi^\ell e_\ell$ such that $\phi(0)=0$, $\partial_k\phi^\ell(0)=1$, then
combined with our knowledge of $g$ we may calculate $d^*d\phi$ and find
\[2\tr (B\otimes d\phi)=2g^{ik}e_\ell(B_{i})\]
from \eqref{eq_connect1}. This yields $B$ and therefore the connection $\nabla$. Subtracting $P\phi$ from \eqref{eq_connect1} then gives $A$, as needed.
\end{proof}
By Propositions \ref{prop_heatkercor}, \ref{prop_transmutecor},  \ref{prop_rDistcor}, and \ref{prop_rConn}, we have recovered the full structure-preserving isomorphism $\Psi$.
\printbibliography[heading=bibnumbered]
\end{document}